\documentclass[12pt]{article}
\usepackage{amsmath,amsfonts,amssymb,amsthm}
\usepackage{bbm}
\usepackage{epsf,epsfig,psfrag,color}
\usepackage{verbatim}
\usepackage[ansinew]{inputenc}

\usepackage{booktabs}

\DeclareMathOperator{\dist}{dist}

\DeclareMathOperator{\diam}{diam}

\DeclareMathOperator{\co}{co}
\DeclareMathOperator{\Proj}{Proj}

\DeclareMathOperator{\argmax}{argmax}

\DeclareMathOperator{\wlim}{w-lim}

\begin{document}
\newtheorem{oberklasse}{OberKlasse}
\newtheorem{lemma}[oberklasse]{Lemma}
\newtheorem{proposition}[oberklasse]{Proposition}
\newtheorem{theorem}[oberklasse]{Theorem}
\newtheorem{remark}[oberklasse]{Remark}
\newtheorem{corollary}[oberklasse]{Corollary}
\newtheorem{definition}[oberklasse]{Definition}
\newtheorem{example}[oberklasse]{Example}
\newtheorem{observation}[oberklasse]{Observation}
\newcommand{\clconvhull}{\ensuremath{\overline{\co}}}
\newcommand{\R}{\ensuremath{\mathbbm{R}}}
\newcommand{\N}{\ensuremath{\mathbbm{N}}}
\newcommand{\Z}{\ensuremath{\mathbbm{Z}}}
\newcommand{\ClSets}{\ensuremath{\mathcal{A}}}
\newcommand{\CpSets}{\ensuremath{\mathcal{C}}}
\newcommand{\CoCpSets}{\ensuremath{\mathcal{CC}}}
\newcommand{\powerset}{\ensuremath{\mathcal{P}}}
\newcommand{\mc}{\mathcal}
\newcommand{\tc}{\textcolor}
\newcommand{\cL}{{\mathcal L}}

\renewcommand{\phi}{\ensuremath{\varphi}}
\renewcommand{\epsilon}{\ensuremath{\varepsilon}}

\title{On the solvability of relaxed one-sided Lipschitz inclusions in Hilbert spaces}
\author{Janosch Rieger and Tobias Weth}
\date{\today}
\maketitle

\begin{abstract}
We prove solvability theorems for relaxed one-sided Lipschitz multivalued mappings
in Hilbert spaces and for composed mappings in the Gelfand triple setting. 
From these theorems, we deduce properties of the inverses of such mappings and
convergence properties of a numerical scheme for the solution of algebraic inclusions.
\end{abstract}

{\bf Key words:} relaxed one-sided Lipschitz property, algebraic inclusion, root-finding method, set-valued analysis.

{\bf AMS(MOS) subject classifications:} 47H04, 65K10, 49J53.\\

\section{Introduction}

The relaxed one-sided Lipschitz property (see Definition \ref{definition:rosl} below)
was first considered in \cite{Donchev:96}, where it
was identified as an important stability criterion for time-dependent differential inclusions.
The behavior of general multivalued mappings with negative relaxed one-sided Lipschitz constants was later studied in \cite{Donchev:02,Donchev:04}. In particular, surjectivity of the mappings and therefore solvability of the corresponding algebraic inclusions was shown 
by considering the flow of the differential inclusions from \cite{Donchev:96}.
However, no information on the localization of the solutions was given in these papers.

For relaxed one-sided Lipschitz mappings in finite-dimensional spaces, the solvability theorem  
given in \cite{Beyn:Rieger:10} specifies a ball in which a solution of the inclusion is contained.
The radius of this ball depends on the norm of the residual of the inclusion at the center point.
This theorem guarantees that the implicit Euler scheme for stiff ordinary differential inclusions
is well-defined and convergent on the infinite time interval, and it has recently been applied in 
\cite{Mordukhovich:Tian:14} to obtain a numerical method for the solution of the generalized Bolza problem.
A refined solvability result presented in \cite{Beyn:Rieger:14} and restated as Theorem \ref{solvability:theorem} 
below immediately gives rise to a numerical algorithm for the solution of algebraic inclusions.

These solvability theorems are relevant for the following reason.
For non-scalar mappings, it is currently unclear whether continuous and relaxed one-sided Lipschitz multivalued mappings 
possess parameterizations that are continuous and one-sided Lipschitz with the same one-sided Lipschitz constant.
Moreover, simple examples 
show that selections generated by metric projection such as the minimal
selection are not one-sided Lipschitz with the same constant as the multimap.
It is therefore impossible to obtain precise solvability results by applying standard tools like
topological fixed point theorems to selections or parameterizations of one-sided Lipschitz multifunctions.

\medskip

In the present paper we generalize the finite-dimensional solvability result 
from \cite{Beyn:Rieger:14} to infinite-dimensional Hilbert spaces, and we discuss implications both in an abstract framework as well as in the 
context of  a special class of systems of elliptic differential inclusions.
After collecting definitions and preliminary tools in Section \ref{section:preliminaries},
we prove an abstract solvability result in an infinite-dimensional Hilbert space in Section \ref{section:solvability} via an approach based on Galerkin
approximations. The approach  avoids strong compactness assumptions which are not satisfied in many applications.
In Section \ref{sec:Gelfand}, we reformulate the main result in the context of Gelfand triples and composed multivalued operators.
In this setting, special care was taken to obtain optimal estimates by considering 
a mixed scalar product adapted to the properties of the individual operators.
As a byproduct, the main result reveals certain aspects of the behavior of the inverses of relaxed one-sided Lipschitz mappings
as detailed in Section \ref{section:inverses}.
As in the finite-dimensional context, the solvability theorem gives rise to a numerical algorithm
for the solution of relaxed one-sided Lipschitz algebraic inclusions, which is analyzed in Section \ref{section:solver}.
In Section~\ref{section:example}, we discuss a system of elliptic differential inclusions where
the assumptions of the Gelfand triple version of our main result are verified for suitable right-hand sides. Moreover, we test the numerical algorithm from Section~\ref{section:solver} in the context of this system.

\section{Preliminaries} \label{section:preliminaries}

In this section, we collect the necessary definitions and some elementary facts. Let $(X,\|\cdot\|_X)$ be any real normed vector space, and let $\langle \cdot,\cdot\rangle : X^* \times X \to \R$ denote 
the dual pairing.  

\begin{definition}
For $x\in X$ and nonempty subsets $M,M'\subset X$, we set
\begin{align*}
\dist_X(x,M') &:= \inf_{x'\in M'}\|x-x'\|_X,\\
e_X(M,M') &:= \sup_{x\in M}\dist_X(x,M'),\\
\|M\| &:= e_X(M,\{0_X\}),\\
\Proj_X(x,M) &:= \{x'\in M: \|x-x'\|_X = \dist_X(x,M)\},\\
B_X(M,R)&:=\{y\in X: \dist_X(y,M) \le R\},\\
B_X(x,R)&:= \{y\in V: \|y-x\|_X \le R\}.
\end{align*}
The nonempty closed, bounded and convex subsets of $X$ are denoted $\mc{CBC}(X)$,
and the nonempty convex and compact subsets of $X$ are denoted $\mc{CC}(X)$.
\end{definition}

\begin{definition}
  \begin{itemize}
  \item[a)]  The support function $\sigma_X^*: X\times\mc{CBC}(X^*)\rightarrow\R$ is defined by
$$
\sigma_X^*(x,A) := \sup_{\phi\in A}\langle\phi,x\rangle\quad \forall x\in X, A\subset X^*.
$$
\item[b)] If $X$ is a real Hilbert space with scalar product $(\cdot,\cdot)_X$, then we define 
$\sigma_X: X\times\mc{CBC}(X)\rightarrow\R$ by 
$$
\sigma_X(x,A) := \sup_{y \in A}(y,x) \quad \forall x\in X, A\subset X.
$$
\end{itemize}
\end{definition}

\begin{definition}
Let $(M,d)$ be a metric space  and $Y$ a further normed vector space.  
\begin{itemize}
\item [a)]  A set-valued mapping $F:M  \rightarrow\mc{CBC}(Y^*)$ is called upper hemicontinuous (uhc) 
at $x\in M$ if for any sequence $(x_k)_{k\in\N}\subset M $ with $x_k\rightarrow x$
we have
\begin{equation} \label{uhc:weak}
\limsup_{k\rightarrow\infty}\sigma_Y^*(v,F(x_k)) \le \sigma_Y^*(v,F(x)) \qquad \text{for all $v\in Y$.}
\end{equation}
It is called uhc if it is uhc at any $x\in M$.\\ 
If $M$ is weakly sequentially closed, then $F$ is called compactly upper hemicontinuous (c-uhc) 
if condition \eqref{uhc:weak} holds for any sequence $(x_k)_{k\in\N}\subset M $ with $x_k\rightharpoonup x\in M$.
\item [b)]  If $Y$ is a Hilbert space, then a set-valued mapping $F:M  \rightarrow\mc{CBC}(Y)$ is called upper hemicontinuous (uhc) 
at $x\in M$ if for any sequence $(x_k)_{k\in\N}\subset M $ with $x_k\rightarrow x$
we have
\begin{equation} \label{uhc:weak-1}
\limsup_{k\rightarrow\infty}\sigma_X(v,F(x_k)) \le \sigma_X(v,F(x)) \qquad \text{for all $v\in Y$.}
\end{equation}
It is called uhc if it is uhc at any $x\in M$.\\ 
If $M$ is weakly sequentially closed, then $F$ is called compactly upper hemicontinuous (c-uhc) 
if condition \eqref{uhc:weak-1} holds for any sequence $(x_k)_{k\in\N}\subset M $ with $x_k\rightharpoonup x\in M$.
\item[c)] A set-valued mapping $F: M \rightarrow\mc{CC}(Y)$ is called upper semicontinuous (usc) at $x\in M$ if for any sequence 
$(x_k)_{k\in\N}\subset M$ with $x_k\rightarrow x \in M$ we have
\begin{equation} \label{usc:condition}
e_V(F_N(x_k),F_N(x)) \rightarrow 0\ \text{as}\ k\rightarrow\infty.
\end{equation}
It is called usc if it is usc at any $x\in M$. 
\end{itemize}
\end{definition}

%
%

The following one-sided property is the central object of investigation in the present paper.
\begin{definition} \label{definition:rosl}
A mapping $F:X\rightrightarrows X^*$ is called $l$-relaxed one-sided Lipschitz with constant $l\in\R$
(or $l$-ROSL) if for any $x,x'\in X$ and $y\in F(x)$ there exists $y'\in F(x')$ such that
\[\langle y'-y,x'-x\rangle \le l\|x'-x\|_X^2.\]
\end{definition}
In Theorems \ref{solvability:hilbert}, \ref{solvability:theorem} and \ref{W:theorem}, the relaxed one-sided Lipschitz property 
will only be required relative to one point in the graph of $F$. In Sections \ref{section:inverses} and \ref{section:solver},
however, we will deal with mappings that are relaxed one-sided Lipschitz in the sense of Definition \ref{definition:rosl}.

\begin{remark} \label{monotone}
The definition of an ROSL mapping with constant $l \ge 0$ is formally similar to that of a monotone mapping.
Nevertheless, ROSL and monotone mappings have fundamentally different properties. Monotone mappings on Hilbert spaces are single-valued outside a set of first Baire category (see \cite{Kenderov:76}), 
and the operator $I+\alpha T$ is onto and possesses a single-valued inverse for any monotone $T$ and any $\alpha>0$,
which is the theoretical basis for the proximal point algorithm (see \cite{Rockafellar:76}). 
In contrast, the $1$-ROSL mapping $F:\R\rightarrow\mc{CC}(\R)$ given by $F(x)=x+[-1,+1]$
is set-valued on the whole space, and for all $\alpha>0$ the inverse of $I+\alpha F$, given by 
$(I+\alpha F)(x)=(1-\alpha)x+[-\alpha,\alpha]$, is set valued as well. 
Similarly, ROSL mappings with constants $l \le 0$, which are mainly considered in the present paper, look formally similar to mappings $T$ with $-T$ monotone but 
have fundamentally different properties. 
\end{remark}

We also recall the following facts which are well known and easy to see. 

\begin{lemma} \label{projections}
Let $X$ be a Hilbert space, and let $x\in X$.
\begin{itemize}
\item [a)] If $M\subset X$ is closed and convex, then $\Proj(x,M)$ is a single point.
\item [b)] If $M\subset X$ is weakly sequentially closed, then $\Proj(x,M)$ is nonempty.
\item [c)] If $M\subset X$ is closed, and if $(x_n)_n\subset X$
and $\bar x\in X$ satisfy $\|x_n-\bar x\|_X\rightarrow 0$ and $\dist_X(x_n,M)\rightarrow 0$
as $n\rightarrow\infty$, then $\bar x\in M$.
\end{itemize}
\end{lemma}

The following standard observations will also be used later on.

\begin{lemma}
\label{CBC-composition}  
Let $Y$ be a reflexive Banach space, $Z$ a normed vector space and $T \in \cL(Y,Z)$.  Then for every 
$A \in \mc{CBC}(Y)$ we have $T(A) \subset \mc{CBC}(Z)$. Consequently, every map $F:M \to \mc{CBC}(Y)$ defined on an arbitrary set $M$ gives rise to a map 
$T \circ F: M \to \mc{CBC}(Z)$ 
\end{lemma}

\begin{proof}
Let $A \in \mc{CBC}(Y)$.  Since $T$ is linear and continuous, $T(A)$ is convex and bounded. To see that $T(A)$ is closed, 
we consider a sequence $(z_k)_k$ in $T(A)$ such that $z_k \to z \in Z$ as $k \to \infty$. 
Choosing $y_k \in A$ such that $T(y_k)=z_k$ for $k \in \N$, we obtain a bounded sequence $(y_k)_k$ in $A$. Since $Y$ is reflexive,
we may pass to a subsequence such that $y_k\rightharpoonup y \in Y$ as $\N'\ni k\rightarrow\infty$, and $y \in A$ by Mazur's
Theorem. 
Moreover, 
$$
T (y)= \wlim_{k \to \infty} T y_k= \wlim_{k \to \infty} z_k = z
$$
and thus $z \in T(A)$. Hence $T(A)$ is closed.
\end{proof}

\begin{lemma} \label{Minkowski}
Let $X$ be a reflexive Banach space, and let $M',M''\in\mc{CBC}(X)$.
Then 
\[M'+M'':=\{m'+m'':m'\in M,m''\in M''\}\in\mc{CBC}(X).\]
\end{lemma}
\begin{proof}
It is easy to check that $M'+M''$ is bounded and convex. 
We show that $M'+M''$ is closed. 
Let $(m_n)_n\subset M'+M''$ be any sequence with $\lim_{n\rightarrow\infty}m_n=m\in X$.
Then there exist $(m_n')_n\subset M'$ and $(m_n'')_n\subset M''$ such that $m_n=m_n'+m_n''$ for all $n$.
By the Banach-Alaoglu theorem, we have $m_n'\rightharpoonup m'\in M'$ along a subsequence,
so that $m_n''=m_n-m_n'\rightharpoonup m-m'$. 
By Mazur's lemma, $m'':=m-m'\in M''$. Therefore, $m=m'+m''\in M'+M''$.
\end{proof}

\section{An infinite-dimensional solvability theorem}  \label{section:solvability}

Let $V$ be a separable Hilbert space with scalar product $(\cdot,\cdot)_V$, associated norm $\|\cdot\|_V$. The main result of this section is the following 
solvability theorem. 

\begin{theorem} \label{solvability:hilbert}
Let $\tilde x, \bar y \in V$, $R>0$ and $l<0$, and let $F:B_V(\tilde x,R)\rightarrow\mc{CBC}(V)$ 
be a multivalued mapping.
\begin{itemize}
\item [a)] Let $F$ be bounded and c-uhc. If there exists some $\tilde y\in F(\tilde x)$ such that 
$\|\bar y-\tilde y\|_{V} \le -lR$ and
\[\forall x\in B_V(\tilde x,R)\ \exists y\in F(x):\ (y-\tilde y,x-\tilde x)_V \le l\|x-\tilde x\|_V^2,\]
then there exists some 
\[\bar x\in B_V(x_c,-\tfrac{1}{2l}\|\tilde y-\bar y\|_{V})\quad \text{with}\quad x_c=\tilde x+\tfrac{1}{2l}(\bar y-\tilde y),\] 
satisfying $\bar y\in F(\bar x)$.
\item [b)] If $F$ admits a modulus of continuity relative to $\tilde x$ in the sense that
\begin{equation*}
e_{V}(F(x),F(\tilde x)) \le \omega(\|x-\tilde x\|_V)
\end{equation*}
for all $x\in B_V(\tilde x,r)$ and some $r\le\min\omega^{-1}(\dist_{V}(\bar y,F(\tilde x)))$, 
then
\begin{equation*}
\|x-\tilde{x}\|_V \ge r\ \text{for all}\ x\in F^{-1}(\bar{y})\cap B_V(\tilde x,R).
\end{equation*}
\end{itemize}
\end{theorem}

A variant of part a) for finite-dimensional Hilbert spaces has been proved in \cite{Beyn:Rieger:14}. 
This variant will be used in the proof of Theorem~\ref{solvability:hilbert} together with a Galerkin approximation.  
For this we let $(w_k)_{k=1}^\infty$ denote an orthonormal basis of $V$, and for $N \in \N$ we consider the finite-dimensional
subspace $V_N: =\text{span}\{w_1,\ldots,w_N\} \subset V$. 
The orthogonal projection from $V$ onto $V_N$ is denoted $P_N$. 
We will then use the following reformulation of \cite[Theorem 3.1]{Beyn:Rieger:14}.

\begin{theorem} \label{solvability:theorem}
Let $\tilde x,\bar y\in V_N$, $R>0$ and $l<0$, let $F_N:B_{V_N}(\tilde x,R)\rightarrow\mc{CC}(V_N)$ 
be a multivalued mapping, and let $\tilde y\in F_N(\tilde x)$.
If $F_N$ is usc, if $\|\bar y-\tilde y\|_V\le-lR$ and if for every 
$x\in B_{V_N}(\tilde x,R)$ there exists some $y\in F_N(x)$ such that
\begin{equation*}
(y-\tilde y,x-\tilde x)_V \le l\|x-\tilde x\|_V^2,
\end{equation*}
then there exists some $\bar x\in B_{V_N}(x_c,-\tfrac{1}{2l}\|\bar y-\tilde y\|_{V})$ with
$x_c:=\tilde x+\tfrac{1}{2l}(\bar y-\tilde y)$ and $\bar y\in F_N(\bar x)$.
\end{theorem}

The remainder of this section is devoted to the 

\begin{proof}[Proof of Theorem~\ref{solvability:hilbert}]$ $\\ 
\emph{Statement a), special case:}
We begin with the case $\tilde x=0$, $\bar y=0$, 
in which $\|\tilde y\|_{V}\le-lR$,
and define $\tilde y_N := P_N\tilde y$.
Clearly, $B_{V_N}(0,R)\subset B_V(0,R) \subset V$ and $\|\tilde y_N\|_V \le -lR$,
and the mapping $F_N:B_{V_N}(0,R)\rightrightarrows V_N$ given by $F_N(x):=P_NF(x)$ is well-defined
with $\tilde y_N\in F_N(0)$. 
Moreover, $F_N$ satisfies the assumptions of Theorem \ref{solvability:theorem} on $B_{V_N}(0,R)$ 
with data $\tilde x_N=0$, $\bar y_N=0$ and $\tilde y_N$:
\begin{itemize}
\item [i)] The mapping $F_N$ is bounded on $B_{V_N}(0,R)$ with convex and compact values: 
As $V_N$ is finite-dimensional and $P_N \in \cL(V,V_N)$, this follows from Lemma~\ref{CBC-composition}.
\item [ii)] The mapping $F_N$ satisfies an ROSL-type condition: Let $x\in B_{V_N}(0,R)$. By assumption, there exists
some $y\in F(x)$ such that $(y-\tilde y,x) \le l\|x\|_V^2$, which implies that 
\[(P_N y-\tilde y_N, x)_V = (y-\tilde y, P_N x)_V = (y-\tilde y, x)_V  \le l\|x\|_V^2.\]
\item [iii)] The mapping $F_N$ is usc: 
Assume that $F_N$ is not usc at $x\in B_{V_N}(0,R)$. Then there exist $\epsilon>0$ and two sequences 
$(x_k)_{k\in\N}\subset B_{V_N}(0,R)$ and $(y_k)_{k\in\N}\subset V_N$ with 
$\lim_{k\rightarrow\infty} x_k = x$ and $y_k\in F_N(x_k)$ 
such that 
\[\dist_{V}(y_k,F_N(x))\ge\epsilon.\]
As a consequence of \cite[Theorem 13.1]{Rockafellar:70} there exist $(v_k)_{k\in\N}\subset V_N$ 
such that $\|v_k\|_{V}=1$ and 
\[(v_k,y_k)_V \ge \sigma_{V_N}(v_k,B_{V_N}(F_N(x),\epsilon)) = \sigma_{V_N}(v_k,F_N(x)) + \epsilon.\]
As $\dim V_N<\infty$, there exists $v\in V_N$ such that $\|v_k-v\|_V\rightarrow 0$ along a subsequence.
Hence
\begin{equation} \label{local:2}
\begin{aligned}
&(v,y_k)_V = (v-v_k,y_k)_V + (v_k,y_k)_V\\
&\ge -\sup_{x'\in B_{V_N}(0,R)}\|F_N(x')\|_V\|v-v_k\|_V + \sigma_{V_N}(v_k,F_N(x)) + \epsilon\\
& \rightarrow \sigma_{V_N}(v,F_N(x)) + \epsilon 
\end{aligned}
\end{equation}
as $k\rightarrow\infty$, because $\sup_{x'\in B_{V_N}(0,R)}\|F_N(x')\|_V<\infty$ according to a).
On the other hand, since $F$ is c-uhc, $F$ is uhc, and for any $v\in V_N$, we have
\begin{align*}
\limsup_{k\rightarrow\infty}\sigma_{V_N}(v,F_N(x_k)) 
&= \limsup_{k\rightarrow\infty}\sup_{y\in F(x_k)}(P_Ny,v)_V
= \limsup_{k\rightarrow\infty}\sup_{y\in F(x_k)}(y,v)_V\\
&= \limsup_{k\rightarrow\infty}\sigma_V(v,F(x_k))
\le \sigma_V(v,F(x))\\ 
&= \sup_{y\in F(x)}(y,v)_V
= \sup_{y\in F(x)}(P_Ny,v)_V
= \sigma_{V_N}(v,F_N(x)).
\end{align*}
This contradicts \eqref{local:2}, and hence $F_N$ is uhc.
\end{itemize}
Therefore, Theorem \ref{solvability:theorem} yields for every $N\in\N$ some 
$\bar x_N \in B_{V_N}(-\frac{1}{2l}\tilde y_N,-\frac{1}{2l}\|\tilde y_N\|_V)$
with $0\in F_N(\bar x_N)$. 
Rewrite $\bar x_N = -\frac{1}{2l}(\tilde y_N+v_N)$ with $\|v_N\|_V\le\|\tilde y_N\|_{V}\le\|\tilde y\|_{V}$.
Then there exists some $\bar v\in V$ with $\|\bar v\|_V\le\|\tilde y\|_{V}$ and such that $v_N\rightharpoonup \bar v$
along a subsequence $\N'\subset\N$. Since moreover $\tilde y_N\rightarrow \tilde y$ as $N\rightarrow\infty$, we infer that 
$$
\bar x_N \rightharpoonup \bar x:=-\tfrac{1}{2l}(\tilde y+\bar v)\in B_V(-\tfrac{1}{2l} \tilde y,-\tfrac{1}{2l}\|\tilde y\|_{V})\
 \text{as}\  \N'\ni N\rightarrow\infty.
$$
Furthermore, since $0\in F_N(\bar x_N)$, there exist elements $\phi_N\in F(\bar x_N)$ with  $P_N\phi_N=0$ for $N \in \N$, which implies that $(\phi_N, w_k) \to 0$ as $N \to \infty$ for every $k \in \N$. Since $(w_k)_k$ is an orthonormal basis of $X$ and the sequence, $(\phi_N)_N$ is bounded as $F$ is bounded, it follows that $\phi_N \rightharpoonup 0$ as $N \to \infty$. For arbitrary $v\in V$, we thus find that 
\[0 = \lim_{N\rightarrow\infty}\langle\phi_N,v\rangle \le \limsup_{N\rightarrow\infty}\sigma_V(v,F(\bar x_N)) 
\le \sigma_V(v,F(\bar x)),\]
because $F$ is c-uhc. This implies $0\in F(\bar x)$.

\medskip

\emph{Statement a), general case:} Consider $\tilde x\in V$, $\bar y\in V$ 
and the map $G:B_V(0,R)\rightarrow\mc{CBC}(V)$ given by
\[G(z):=F(z+\tilde x)-\bar y,\]
Clearly $y_0:=\tilde y-\bar y\in F(\tilde x)-\bar y = G(0)$. 
For any $x\in B_V(\tilde x,R)$, there exists $z\in B_V(0,R)$ such that $z+\tilde x=x$, and
by assumption, there exists some $y\in F(x)$ such that 
\[( y-\tilde y,x-\tilde x)_V \le l\|x-\tilde x\|_V^2 = l\|z\|_V^2.\]
But then $y':=y-\bar y\in F(x)-\bar y = G(z)$ satisfies
\begin{align*}
( y'-y_0,z)_V = \bigl( (y-\bar y)-(\tilde y-\bar y),x-\tilde x\bigr)_V = ( y-\tilde y,x-\tilde x)_V 
\le l\|z\|_V^2,
\end{align*}
so that $G$ satisfies all assumptions of Step 1,
which guarantees the existence of some $z_0\in B_V(0,R)$ with $0\in G(z_0)$ and
\[\|z_0+\tfrac{1}{2l}y_0\|_V \le -\tfrac{1}{2l}\|y_0\|_{V}.\]
Setting $\bar x:=\tilde x+z_0$ we obtain $\bar y\in F(\bar x)$ and 
\begin{align*}
\|\bar x-x_c\|_V &= \|\bar x -\tilde x-\tfrac{1}{2l}(\bar y-\tilde y)\|_V = \|z_0+\tfrac{1}{2l}y_0\|_V 
\le -\tfrac{1}{2l}\|y_0\|_{V} = -\tfrac{1}{2l}\|\tilde y-\bar y\|_{V}.
\end{align*}
\emph{Statement b):} Assume that $\bar y\in F(x)$ for some $x\in B_V(\tilde x,r)$.
Then
\[ \dist_{V}(\bar y,F(\tilde x)) \le e_{V}(F(x),F(\tilde x)) \le \omega(\|x-\tilde x\|_V)\]
implies 
\[\|x-\tilde x\|_V \ge \min\omega^{-1}(\dist_{V}(\bar y,F(\tilde x))) \ge r.\]
\end{proof}

\section{A reformulation for Gelfand triples} \label{sec:Gelfand}

The aim of this section is to adapt the above solvability theorem to a situation in which
the multivalued operator consists of two parts with different properties.
Theorem \ref{W:theorem} improves the approach presented in \cite{Beyn:Rieger:12}
by considering the problem in a space that is adapted to the composed operator.
We postpone a comparison of both results to Remark \ref{explanation} at the end of this section.
The most prominent setting, in which such a splitting occurs, will be discussed in the extended 
example in Section \ref{section:example}. 

Let $(V,\|\cdot\|_V,(\cdot,\cdot)_V)$ and $(H,\|\cdot\|_H,(\cdot,\cdot)_H)$ 
be separable Hilbert spaces such that $V$ is densely and continuously embedded into $H$ with embedding constant $c_{V\!H}>0$.
Identifying $H$ with its dual $H^*$, we then have embeddings
$$
V \overset{i}{\hookrightarrow} H \overset{i^*}{\hookrightarrow} V^*.
$$
Here $i^*$ denotes the dual of $i$, and this map is injective due to the density of $i(V)$ in $H$. As usual, we regard $V$ as a
subspace of $H$ and $H$ as a subspace of $V^*$, writing simply $v \in H$ instead of $i(v)$ and $w \in V^*$ instead of $i^*(w)$ for
$v \in V$, $w \in H$.   With these simplifications, we have  
$$
\langle y,x\rangle = (y,x)_H\quad \forall x\in V,\ y\in H.
$$
In the following, we fix constants 
\begin{equation}
  \label{eq:ineq-constants-lVlH}
l_V< 0\qquad \text{and}\qquad  l_H < -l_V/c_{VH}^2.  
\end{equation}
 Then the bilinear form 
$$
(x_1,x_2) \mapsto (x_1,x_2)_W:= -l_V (x_1,x_2)_V -l_H (x_1,x_2)_H,\qquad x_1,x_2 \in V
$$
is a scalar product which induces an equivalent norm $\|\cdot\|_W$ on $V$. In the following, $B_W(x,R)$ denotes the ball with radius $R>0$ w.r.t.  $\|\cdot\|_W$ centered at $x \in V$. Moreover, for $y \in V^*$, we denote by $\|y\|_{W^*}$ the dual norm induced by $\|\cdot\|_W$, i.e., 
$\|y\|_{W^*} = \sup \limits_{x \in V \setminus \{0\}}\frac{\langle y,x \rangle}{\|x\|_W}$ for $y \in V^*$. We also denote by $J_W:V^*\rightarrow V$ the corresponding canonical isometric isomorphism given by 
$$
(J_W\phi,v)_W = \langle\phi,v\rangle\quad \forall v\in V,\: \phi \in V^*.
$$
The following theorem is a variant of Theorem~\ref{solvability:hilbert} for composite operators.

\begin{theorem} \label{W:theorem}
Suppose that $l_V,l_H\in\R$ satisfy (\ref{eq:ineq-constants-lVlH}), and let $\tilde x\in V$, $\bar y\in V^*$ and $R>0$.
Moreover, let 
$$
F_V:B_W(\tilde x,R) \subset V \rightarrow\mc{CBC}(V^*)\quad \text{and}\quad F_H:B_W(\tilde x,R) \subset V \rightarrow\mc{CBC}(H)
$$ 
be bounded and c-uhc, and let $F: B_W(\tilde x,R) \rightrightarrows V^*$ be given by $F=F_V+ F_H$, i.e. 
$$
F(v) := \{y_V+y_H: y_V\in F_V(v), y_H\in F_H(v)\} \quad \text{for}\quad v \in V.
$$
Suppose furthermore that there exists 
$\tilde y_V\in F_V(\tilde x)$, $\tilde y_H\in F_H(\tilde x)$ such that 
$$
\|\bar y-\tilde y\|_{W^*}\le R 
$$
with $\tilde y:=\tilde y_V+\tilde y_H$ and 
\begin{equation} \label{partly:rosl}
\left. 
 \begin{aligned}
& \forall x\in B_W(\tilde x,R)\\
&\exists y_V\in F_V(x), y_H\in F_H(x)   
  \end{aligned}
 \right\}
\:\text{with}\:
\left \{
\begin{aligned}
&\langle y_V-\tilde y_V,x-\tilde x\rangle\le l_V\|x-\tilde x\|_V^2;\\
&(y_H-\tilde y_H,x-\tilde x)_H \le l_H\|x-\tilde x\|_H^2.
\end{aligned}
\right.
\end{equation}
Finally, let $x_c:=\tilde x-\tfrac{1}{2}J_W(\bar y-\tilde y)$. Then there exists some 
\begin{equation}
  \label{eq:estimate-gelfand-theorem}
\bar x\in B_W(x_c,\tfrac{1}{2}\|\tilde y-\bar y\|_{W^*})\quad \text{satisfying}\quad \text{$\bar y\in F(\bar x)$.}
 \end{equation}
\end{theorem}

\begin{remark}
\label{sec:reform-gelf-tripl}{\rm 
(i) In the case where the embedding of $V$ in $H$ is compact, it suffices to assume that 
$F_H:B_W(\tilde x,R) \subset H \rightarrow\mc{CBC}(H)$ is bounded and uhc, because then $F_H:B_W(\tilde x,R) \subset V \rightarrow\mc{CBC}(H)$ is bounded and c-uhc.\\
 (ii) The assumption  (\ref{partly:rosl}) arises naturally in applications, and it is the reason for using the mixed norm $\|\cdot\|_W$ which then gives rise to optimal estimates, as explained in (iii) below.  Nevertheless, sufficient assumptions can easily be formulated in terms of $\|\cdot\|_V$ and $\|\cdot\|_H$ by using the estimates 
 \begin{equation*}
\left\{  
 \begin{aligned}
 &-l_V \|x\|_V^2 \le \|x\|_W^2 \le - (l_V + c_{VH}^2 l_H) \|x\|_V^2   &&\quad \text{if $l_H \le 0$,}\\
&- (l_V + c_{VH}^2 l_H) \|x\|_V^2 \le \|x\|_W^2   \le -l_V \|x\|_V^2 &&\quad \text{if $l_H \in [0,-\frac{l_V}{c_{VH}^2})$}   
\end{aligned}
\right.
 \end{equation*}
for $x \in V$ and 
\begin{equation}
   \label{eq:est-equiv-norms-2}
\left\{  
 \begin{aligned}
 &-\frac{1}{l_V} \|y\|_{V*}^2 \le \|y\|_{W^*}^2 \le - \frac{1}{l_V + c_{VH}^2 l_H} \|y\|_{V^*}^2   &&\quad \text{if $l_H \in [0,-\frac{l_V}{c_{VH}^2})$,}\\
&- \frac{1}{l_V + c_{VH}^2 l_H} \|y\|_{V^*}^2 \le \|y\|_{W^*}^2  \le -\frac{1}{l_V} \|y\|_{V^*}^2    &&\quad \text{if $l_H \le 0$}
 \end{aligned}
\right.
 \end{equation}
for $y \in V^*$.\\
(iii) In the case where $l_H \in [0,-\frac{l_V}{c_{VH}^2})$, (\ref{eq:estimate-gelfand-theorem}) and (\ref{eq:est-equiv-norms-2}) imply the estimate 
\begin{equation*} 
- l_V \|\bar x-x_c\|_V^2-l_H\|\bar x -x_c\|_H^2 \le -\frac{1}{4(l_V + c_{VH}^2 l_H)} \|\bar y-\tilde y\|_{V^*}^2. 
\end{equation*}
and therefore 
\begin{equation*} 
\|\bar x-x_c\|_H \le \tfrac{c_{V\!H}}{2(l_V+c_{V\!H}^2l_H)}\|\bar y-\tilde y\|_{V^*}. 
\end{equation*}
In the case where $l_H\le 0$, (\ref{eq:estimate-gelfand-theorem}) and (\ref{eq:est-equiv-norms-2})  imply the estimate 
\begin{equation} 
\label{V:estimate-2}
- l_V \|\bar x-x_c\|_V^2-l_H\|\bar x -x_c\|_H^2 \le -\frac{1}{4 l_V} \|\bar y-\tilde y\|_{V^*}^2. 
\end{equation}
and therefore 
\begin{equation} \label{H:estimate-2}
\|\bar x-x_c\|_H \le \tfrac{c_{V\!H}}{2\sqrt{l_V(l_V+c_{V\!H}^2l_H)}}\|\bar y-\tilde y\|_{V^*}. 
\end{equation}
Hence a negative one-sided Lipschitz constant $l_H$ of $F_H$ improves the estimate for $\|\bar x-x_c\|_H$, which is important in
the case where  $l_H\ll l_V<0$. }
\end{remark}

\begin{proof}[Proof of Theorem~\ref{W:theorem}]
We apply Theorem \ref{solvability:hilbert} to the Hilbert space $(V, (\cdot,\cdot)_W)$, the map $F_0:= J_W \circ F:  
B_W(\tilde x,R)\rightrightarrows V$ in place of $F$ and 
$\bar y_0:= J_W \bar y,\; \tilde y_0:= J_W \tilde y \in V$ in place of $\bar y, \tilde y$, respectively.  We check that $F_0$ satisfies the assumptions of Theorem \ref{solvability:hilbert}. 
\begin{itemize}
\item [i)] \emph{One-sided Lipschitz property.} Let $x\in B_W(\tilde x,R)$.
By \eqref{partly:rosl}, there exist $y_V\in F_V(x)$ and $y_H\in F_H(x)$ such that,
denoting $y:=J_W(y_V+y_H)$, we obtain $y\in F_0(x)$ and
\begin{equation} \label{kleiner:trick}
\begin{aligned}
( y-\tilde y_0 ,x-\tilde x )_W
&=\langle y_V+y_H-\tilde y ,x-\tilde x \rangle\\
&= \langle y_V-\tilde y_V,x-\tilde x\rangle +(y_H-\tilde y_H,x-\tilde x)_H\\
&\le l_V\|x-\tilde x\|_V^2 + l_H\|x-\tilde x\|_H^2 = - \|x-\tilde x\|_W^2
\end{aligned}
\end{equation}
Therefore, $F$ is relaxed one-sided Lipschitz relative to $\tilde x$ and $\tilde y$ on $B_W(\tilde x,R)$ with constant $l=-1$ w.r.t $\|\cdot\|_W$.
\item [ii)] \emph{Properties of the images.} 
For any $x\in B_W(\tilde x,R)$, we have $F_V(x)\in\mc{CBC}(V^*)$ and $F_H(x)\in\mc{CBC}(H)$. Consequently, $F_H(x)\in\mc{CBC}(V^*)$
by Lemma~\ref{CBC-composition}. Then Lemma \ref{Minkowski} guarantees that $F(x) \in\mc{CBC}(V^*)$. Again by Lemma~\ref{CBC-composition} it follows that $F_0(x)= J_W(F(x)) \in \mc{CBC}(V)$.
\item [iii)] \emph{$F_0$ is bounded and c-uhc.} The boundedness of $F_0$ is an easy consequence of the boundedness of the maps $F_V$ and $F_H$.  To show that $F_0$ is c-uhc, let $(x_n)_n\subset B_W(\tilde x,R)$ satisfy $x_n\rightharpoonup x\in B_W(\tilde x,R)$, and let $v \in V$. Using that $F_H$ is c-uhc as a map to $H$ by assumption, we then find that
\begin{align*}
&\limsup_{n\rightarrow\infty}\sigma_V^*(v,F_H(x_n)) 
= \limsup_{n\rightarrow\infty}\sup_{y\in F_H(x_n)}\langle y,v\rangle\\
&= \limsup_{n\rightarrow\infty}\sup_{y\in F_H(x_n)}(y,v)_H
= \limsup_{n\rightarrow\infty}\sigma_H(v,F_H(x_n))\\
&\le \sigma_H(v,F_H(x))
= \sup_{y\in F_H(x)}(y,v)_H\\
&= \sup_{y\in F_H(x)}\langle y,v\rangle
= \sigma_V^*(v,F_H(x)). 
\end{align*}
Combining this with the fact that $F_V:B_W(\tilde x,R) \subset V \rightarrow\mc{CBC}(V^*)$ is c-uhc by assumption, we find that 
\begin{align*}
&
\limsup_{n\rightarrow\infty}\sigma_V(v,F_0(x_n))=\limsup_{n\rightarrow\infty}\sigma_V^*(v,F(x_n))\\
&= \limsup_{n\rightarrow\infty}\sigma_V^*(v,F_V(x_n)+F_H(x_n))\\
&= \limsup_{n\rightarrow\infty}\{\sigma_V^*(v,F_V(x_n))+\sigma_V^*(v,F_H(x_n))\}\\
&\le \limsup_{n\rightarrow\infty}\sigma_V^*(v,F_V(x_n))+\limsup_{n\rightarrow\infty}\sigma_V^*(v,F_H(x_n))\\
&\le \sigma_V^*(v,F_V(x))+\sigma_V^*(v,F_H(x))
= \sigma_V^*(v,F_V(x)+F_H(x))\\
&= \sigma_V^*(v,F(x)) =\sigma_V (v,F_0(x)) ,
\end{align*}
and thus $F_0$ is c-uhc. 
Note that $\sigma_V$ is defined here w.r.t.\ $(\cdot,\cdot)_W$. 
\end{itemize}
As a consequence, Theorem \ref{solvability:hilbert} applies with $l=-1$ and yields the desired statement.
\end{proof}

\begin{remark} \label{explanation}
In \cite{Beyn:Rieger:12}, elliptic partial differential inclusions with ROSL right-hand sides have been considered 
as ROSL operator inclusions. 
Error estimates for Galerkin approximations have been obtained directly without usage of the mixed norm $\|\cdot\|_W$.
In that case, the core of such an estimate is an inequality similar to \eqref{kleiner:trick}, but of the shape
\begin{align*}
(y-\tilde y ,x-\tilde x)_V
&\le l_V\|x-\tilde x\|_V^2 + l_H\|x-\tilde x\|_H^2 \\
&\le l_V\|x-\tilde x\|_V^2 + \max\{0,l_H\}c_{VH}^2\|x-\tilde x\|_V^2,
\end{align*}
because the fixed point argument must be carried out in $V$ where the differential operator is defined.
Therefore, the ROSL constant of $F_H$ cannot be exploited when $l_H<0$, i.e.\ when this is most desirable.
In this situation, considering the inclusion in $V$ equipped with the mixed norm $\|\cdot\|_W$ as above yields 
estimates \eqref{V:estimate-2} and \eqref{H:estimate-2}.
\end{remark}

\section{Inverses of ROSL mappings} \label{section:inverses}

The properties of relaxed one-sided Lipschitz mappings have been studied, e.g., in \cite{Donchev:02}, \cite{Donchev:04}
and other works of the same author.
At that time, no quantitative information about solutions of algebraic inclusions was available,
and thus only qualitaive properties of these mappings and their inverses could be given.
With Theorem \ref{solvability:hilbert} at our disposal, we can now prove some basic properties of the inverses.

\begin{theorem} \label{theorem:inverse}
Let $F:V\rightarrow\mc{CBC}(V)$ be c-uhc, bounded on bounded sets and $l$-relaxed one-sided Lipschitz with $l<0$.
Then its inverse $F^{-1}:V\rightrightarrows V$ has nonempty weakly sequentially closed images, it is $-\tfrac{1}{l}$-Lipschitz, 
and it is $0$-relaxed one-sided Lipschitz.
Moreover, the images of $F^{-1}$ satisfy the explicit and implicit bounds
\begin{align}
\|F^{-1}(y)\|_V &\le -\tfrac{1}{l}\|F(0)\|_{V} -\tfrac{1}{l}\|y\|_{V}, \label{boundF}\\
\diam_V(F^{-1}(y)) &\le -\tfrac{1}{l}\sup_{x\in F^{-1}(y)} \diam_{V} F(x) \label{diameter} 
\end{align}
for any $y\in V$. In particular, $F^{-1}$ is bounded on bounded sets.
\end{theorem}
\begin{proof}
\emph{Properties of the images of $F^{-1}$.}
For arbitrary $y\in V$, apply Theorem \ref{solvability:hilbert} to the data $F$, $\tilde x=0$, $\bar y = y$ 
and arbitrary $\tilde y\in F(0)$ to find $F^{-1}(y)\neq\emptyset$.

Let $y\in V$, $x\in F^{-1}(y)$ and $x'\in V$.   
By the relaxed one-sided Lipschitz property, there exists some $y'\in F(x')$ such that
\[-\|y'-y\|_{V}\|x'-x\|_V \le (y'-y,x'-x)_V \le l\|x'-x\|_V^2\]
and thus
\begin{equation} \label{local:3}
\|x'-x\|_V \le -\tfrac{1}{l}\|y'-y\|_{V}. 
\end{equation}
Considering $x'=0$, we conclude that 
\[\|x\|_V \le -\tfrac{1}{l}\|y'-y\|_{V} \le -\tfrac{1}{l}\|y'\|_{V}-\tfrac{1}{l}\|y\|_{V} 
\le -\tfrac{1}{l}\|F(0)\|_{V}-\tfrac{1}{l}\|y\|_{V},\]
so that bound \eqref{boundF} holds. Moreover, considering $x'\in F^{-1}(y)$, we deduce from inequality \eqref{local:3} that 
\[\|x'-x\|_V \le -\tfrac{1}{l}\|y'-y\|_{V} \le -\tfrac{1}{l}\diam_{V} F(x'),\]
so that estimate \eqref{diameter} holds.

Let $y\in V$ and $x\in V$, and let $(x_k)_{k\in\N}\subset F^{-1}(y)$ be any sequence such that 
$x_k\rightharpoonup x$ as $k\rightarrow\infty$. 
As $y\in F(x_k)$, we have $\langle y,v\rangle\le\sigma_V(v,F(x_k))$ for all $v\in V$ and $k\in\N$,
and since $F$ is c-uhc, it follows that
\[(y,v)_V \le \limsup_{k\rightarrow\infty}\sigma_V(v,F(x_k)) \le \sigma_V(v,F(x))\]
for all $v\in V$, which shows $y\in F(x)$ and hence $x\in F^{-1}(y)$. 
Therefore, $F^{-1}(y)$ is weakly sequentially closed.

\emph{One-sided and Lipschitz estimates.}
Let $y,y'\in V$ and $x\in F^{-1}(y)$. By Theorem \ref{solvability:hilbert}, there exists
some $x'\in F^{-1}(y')$ such that 
\[x'\in B_V(x+\tfrac{1}{2l}(y'-y),-\tfrac{1}{2l}\|y'-y\|_{V}).\]
In particular,
\[\|x'-x\|_V \le -\tfrac{1}{l}\|y'-y\|_{V},\]
so that $F^{-1}$ is $-\tfrac{1}{l}$-Lipschitz.
Writing $x'=x+\tfrac{1}{2l}(y'-y)+v$ with $v\in B_V(0,-\tfrac{1}{2l}\|y'-y\|_{V})$, we find
\begin{equation*}
\begin{gathered}
\begin{aligned} 
(y'-y,x'-x)_V
&= (y'-y,\tfrac{1}{2l}(y'-y)+v)_V\\
&\le \|v\|_V\|y'-y\|_{V} + \tfrac{1}{2l}\|y'-y\|_{V}^2 \le 0,
\end{aligned}
\end{gathered}
\end{equation*}
so that $F^{-1}$ is $0$-relaxed one-sided Lipschitz. 
\end{proof}

Without additional structure, it seems difficult to say more about the properties
of the inverse. 
Some multivalued mappings arising in control theory or from uncertainties are
explicitly given in parameterized form and allow a more detailed analysis. 
The existence of such a parametrization implies weak-strong continuity
of the multifunction and hence is a substantially stronger assumption than the 
compact upper hemicontinuity required in Theorem \ref{solvability:hilbert}.

\begin{proposition}
Let $(U,d_U)$ be a metric space, and let $F:V\rightarrow\mc{CBC}(V)$ be parameterized by a function 
$f:V\times U\rightarrow V$ satisfying
\begin{itemize}
\item [a)] $F(x)=\cup_{u\in U}f(x,u)$ for all $x\in V$, 
\item [b)] $u\mapsto f(x,u)$ is continuous for all $x\in V$, and
\item [c)] $x\mapsto f(x,u)$ is continuous from $V$ endowed with the weak topology to $V$ endowed with
the norm topology, bounded on bounded sets, and $l$-one-sided Lipschitz for all $u\in U$
in the sense that
\[(f(x,u)-f(x',u),x-x')_V \le l\|x-x'\|_V^2\ \text{for all}\ u\in U\]
with a constant $l<0$ that is independent of $u$.
\end{itemize}
Then $F^{-1}:V\rightrightarrows V$ is parameterized by a function $g:V\times U\rightarrow V$ satisfying
\begin{itemize}
\item [d)] $F^{-1}(y)=\cup_{u\in U}g(y,u)$ for all $y\in V$, 
\item [e)] $u\mapsto g(y,u)$ is continuous for all $y\in V$, and
\item [f)] $y\mapsto g(y,u)$ is $-\tfrac{1}{l}$-Lipschitz and $0$-one-sided Lipschitz for all $u\in U$.
\end{itemize}
In particular, if $U$ is connected or path connected, the images of $F$ and $F^{-1}$ inherit these properties.
\end{proposition}
\begin{proof}
Applying Theorem \ref{theorem:inverse} to the functions $x\mapsto f(x,u)$, $u\in U$, yields the existence of
inverses $g(\cdot,u):V\rightrightarrows V$, $u\in U$, given by
\[g(y,u) := \{x\in V: f(x,u)=y\},\]
that are well-defined, $-\tfrac{1}{l}$-Lipschitz and $0$-ROSL, so that f) holds. 
If $y\in V$, $u\in U$ and $x,x'\in g(y,u)$, then
\[0 = (f(x,u)-f(x',u),x'-x)_V \le l\|x-x'\|_V^2,\]
which enforces $x=x'$, so that $g$ is a single-valued function.

Let $y\in V$ and $x\in F^{-1}(y)$. Then $y\in F(x)$, and hence there exists $u\in U$ with $y=f(x,u)$,
so that $x=g(f(x,u),u)=g(y,u)$, so that d) is valid.

Let $y\in V$ and $u\in U$ be arbitrary, and let $(u_n)_n\subset U$ be such that $d_U(u,u_n)\rightarrow 0$
as $n\rightarrow\infty$.
By Lipschitz continuity of $y\mapsto g(y,u_n)$, we have
\begin{align*}
&\|g(y,u)-g(y,u_n)\|_V \le -\tfrac{1}{l}\|f(g(y,u),u_n)-f(g(y,u_n),u_n)\|_{V}\\
&\le -\tfrac{1}{l}\|f(g(y,u),u_n)-f(g(y,u),u)\|_{V} -\tfrac{1}{l}\|f(g(y,u),u)-f(g(y,u_n),u_n)\|_{V}\\
&= -\tfrac{1}{l}\|f(g(y,u),u_n)-f(g(y,u),u)\|_{V} \rightarrow 0\ \text{as}\ n\rightarrow\infty.
\end{align*}
Hence $u\mapsto g(y,u)$ is continuous, and we have verified e).
\end{proof}

\section{Numerical solution of algebraic inclusions} \label{section:solver}

We propose the algorithm given below in \eqref{algorithm:formulation} for the computation of a solution 
of the algebraic inclusion $\bar y \in F(x)$, where $\bar y\in V$ is given and $F$ is $l$-ROSL and $L$-Lipschitz 
with $l<0$ and a moderate Lipschitz constant $L>0$.
According to estimate \eqref{dist:point} below considered for $n=0$, the method, when analyzed without round-off errors, 
i.e.\ with $\xi_n=0$ for all $n\in\N$, finds a solution $\bar x$ with 
\[\|x_0-\bar x\|_V \le -\tfrac{1}{2l+L}\dist_{V}(\bar y,F(x_0)).\]
When applied in this context to the points $\tilde x=x_0$ and $\tilde y=\Proj_V(\bar y,F(x_0))$, 
Theorem \ref{solvability:hilbert} guarantees the existence of a solution $\bar x'$ of the same inclusion with
\[\bar x'\in B_V(x_c,-\tfrac{1}{2l}\dist_{V}(\bar y,F(x_0)))\quad \text{and}\quad 
x_c=x_0+\tfrac{1}{2l}(\bar y-\Proj_V(\bar y,F(x_0))),\] 
which means that 
\[\|x_0-\bar x'\|_V \le -\tfrac{1}{l}\dist_{V}(\bar y,F(x_0)).\]
Therefore, the solution $\bar x$ of the numerical method is up to a factor $\tfrac{1}{2+L/l}$ as close to the 
initial value $x_0$ as the theoretical estimate.
This is not necessarily true for an arbitrary scheme in the set-valued context,
and therefore an interesting feature of this algorithm.

Moreover, it is currently unclear whether a continuous and $l$-ROSL multivalued mapping 
possesses a selection or a parametrization that is continuous and (uniformly) $l$-one-sided Lipschitz.
This means that, in general, it is impossible to apply a standard numerical method to 
a single-valued selection $f$ of $F$ and to compute in this way a solution of the 
multivalued problem. 
The most promising construction in this direction has been published in \cite{Baier:Farkhi:07},
where set-valued mappings were parameterized by generalized Steiner points of their images.

The basic technique behind the following proposition is the same as in \cite[Proposition 4.1]{Beyn:Rieger:14}.
There are, however, some additional difficulties, because the images of $F^{-1}$ are not compact
and iterates cannot be computed exactly in the current setting.
Computational errors will therefore be modelled by a sequence $(\xi_n)_n\subset V$.

\begin{proposition}
Let $\bar y\in V$, and let $F:V\rightarrow\mc{CBC}(V)$ be c-uhc, $L$-Lipschitz and $l$-relaxed 
one-sided Lipschitz with $l<0$ such that $0\le\kappa:=-\tfrac{L}{2l}<1$.
For arbitrary $x_0\in V$, define $(v_n)_n\subset V$ and $(x_n)_n\subset V$ by
\begin{equation} \label{algorithm:formulation}
\begin{aligned}
v_n := \bar y-\Proj_{V}(\bar y,F(x_n)),\quad 
x_{n+1} := x_n+\tfrac{1}{2l}v_n+\xi_n,\quad n\in\N,
\end{aligned}
\end{equation}
where $(\xi_n)_n\subset V$ is an arbitrary sequence such that $\sum_{n=0}^\infty\|\xi_n\|_V<\infty$.
Then the sequence $(\eta_n)_n\subset\R_+$ given by 
\[\eta_n:=\sum_{k=0}^n\kappa^k\|\xi_{n-k}\|_V\]
satisfies $\sum_{n=0}^\infty\eta_n<\infty$, and the sequence 
$(x_n)_n$ converges to some $\bar x\in F^{-1}(\bar y)$ with estimates
\begin{align}
\dist_V(x_n,F^{-1}(\bar y)) &\le -\tfrac{\kappa^{n-1}}{2l}\|v_0\|_{V}+\eta_{n-1}, \label{dist:set}\\
\|x_n-\bar x\|_V &\le -\tfrac{1}{2l}\tfrac{\kappa^n}{1-\kappa}\|v_0\|_{V} + \sum_{j=n}^\infty\eta_{j}. \label{dist:point}
\end{align}
\end{proposition}

\begin{proof}
According to Lemma \ref{projections}, the projection $\Proj_{V}(\bar y,F(x))$ is a singleton for every $x\in V$, so that
the sequences $(v_n)_n$ and $(x_n)_n$ are well-defined.
Note that
\[\sum_{j=0}^\infty\eta_j=(\sum_{j=0}^\infty\kappa^j)(\sum_{j=0}^\infty\|\xi_j\|_V)<\infty\]
by the Cauchy product formula. Applying Theorem \ref{solvability:hilbert} for every $n\in\N$ with $\tilde x = x_{n}$ and $\tilde y = \Proj_{V}(\bar y,F(x_n))$, 
we find $\bar x_n\in F^{-1}(\bar y)$
such that
\begin{equation} \label{dist:est}
\begin{aligned}
&\dist_V(x_{n+1},F^{-1}(\bar y)) \le \|x_{n+1}-\bar x_n\|_V\\ 
&= \|x_n+\tfrac{1}{2l}v_n+\xi_n-\bar x_n\|_V \le -\tfrac{1}{2l}\|v_n\|_{V}+\|\xi_n\|_V.
\end{aligned}
\end{equation}
By Theorem \ref{theorem:inverse}, the preimage $F^{-1}(\bar y)$ is weakly sequentially closed, 
and therefore, by Lemma \ref{projections}, there exist points $\tilde x_n\in F^{-1}(\bar y)$
such that 
\[\|x_{n+1}-\tilde x_n\|_V = \dist_V(x_{n+1},F^{-1}(\bar y))\ \forall n\in\N.\]
It follows from inequality \eqref{dist:est} that
\begin{align*}
\|v_{n+1}\|_{V} &= \dist_{V}(\bar y,F(x_{n+1})) 
\le e_{V}(F(\tilde x_n),F(x_{n+1})) \le L\|\tilde x_n-x_{n+1}\|_V\\
&= L\dist_V(x_{n+1},F^{-1}(\bar y)) \le -\tfrac{L}{2l}\|v_n\|_{V} + L\|\xi_n\|_V = \kappa\|v_n\|_{V} + L\|\xi_n\|_V,
\end{align*}
so that 
\[\|v_n\|_{V} \le \kappa^n\|v_0\|_{V} + L\eta_{n-1},\]
and, again because of \eqref{dist:est}, we have
\begin{align*} 
\dist_V(x_n,F^{-1}(\bar y)) &\le -\tfrac{1}{2l}\|v_{n-1}\|_{V} + \|\xi_{n-1}\|_V \\
&\le -\tfrac{\kappa^{n-1}}{2l}\|v_0\|_{V}+\kappa\eta_{n-2}+\|\xi_{n-1}\|_V
= -\tfrac{\kappa^{n-1}}{2l}\|v_0\|_{V}+\eta_{n-1},
\end{align*}
which is \eqref{dist:set}. 
Then
\begin{align*}
\|x_{n+1}-x_n\|_V &\le -\tfrac{1}{2l}\|v_n\|_{V} + \|\xi_n\|_V \\
&\le -\tfrac{\kappa^n}{2l}\|v_0\|_{V} + \kappa\eta_{n-1} + \|\xi_n\|_V  
\le -\tfrac{\kappa^n}{2l}\|v_0\|_{V} + \eta_n
\end{align*}
implies that for any $n>k$, we have
\begin{equation} \label{cauchy:estimate}
\begin{aligned}
\|x_n-x_k\|_V &\le \sum_{j=k}^{n-1}\|x_{j+1}-x_j\|_V 
\le \sum_{j=k}^{n-1}(-\tfrac{\kappa^j}{2l}\|v_0\|_{V} + \eta_j)\\
&\le -\tfrac{1}{2l}\|v_0\|_{V}\tfrac{\kappa^k}{1-\kappa} + \sum_{j=k}^\infty\eta_{j}
\rightarrow 0\ \text{as}\ k\rightarrow\infty.
\end{aligned}
\end{equation}
In particular, the sequence $(x_n)_n\subset V$ is Cauchy and hence converges to some $\bar x\in V$.
Since $F^{-1}(\bar y)$ is weakly sequentially closed and inequality \eqref{dist:set} holds, 
Lemma \ref{projections} guarantees that $\bar x\in F^{-1}(\bar y)$.
Estimate \eqref{dist:point} follows from \eqref{cauchy:estimate} by passing to the limit $n\rightarrow\infty$.
\end{proof}

\section{Example} \label{section:example}

In this section, we consider a class of systems of elliptic differential inclusions.
Scalar partial differential inclusions with ROSL right-hand sides have been studied in \cite{Beyn:Rieger:12}.
Existence and relaxation theorems have been proved in a more general context. 
For a recent contribution, we refer to \cite{Cheng:Cong:Xue:11}.
Elliptic partial differential inclusions with multivalued mappings given in terms of subdifferentials 
have been studied, e.g., in the monograph \cite{Carl:Le:Motreanu:07}.

\subsection{A system of elliptic differential inclusions} \label{example:theory}

We consider the system
\begin{equation} \label{original:system}
\begin{gathered}
\begin{aligned}
&(-\Delta u_1(x),-\Delta u_2(x)) \in f(x,u_1(x),u_2(x)),\quad x\in\Omega,\\
&u_1(x)=u_2(x)=0,\quad x\in\partial\Omega,
\end{aligned}
\end{gathered}
\end{equation}
of elliptic partial differential inclusions, where $\Omega\subset\R^d$ is a bounded domain, and
let $f:\Omega\times\R^2\rightarrow\mc{CC}(\R^2)$ be a multivalued mapping with the following properties.
\begin{itemize}
\item [A1)] The mapping $f$ is Caratheodory in the sense that
$x\mapsto f(x,s_1,s_2)$ is measurable for any $(s_1,s_2)\in\R^2$ 
and $(s_1,s_2)\mapsto f(x,s_1,s_2)$ is continuous for almost every $x\in\Omega$.
\item [A2)] The mapping $f$ is uniformly $l_f$-ROSL in the sense that for almost every $x\in\Omega$
and every $s=(s_1,s_2)\in\R^2$, $t=(t_1,t_2)\in\R^2$ and $\eta=(\eta_1,\eta_2)\in f(x,s_1,s_2)$, there exists
$\rho=(\rho_1,\rho_2)\in f(x,t_1,t_2)$ with
\[\langle \rho-\eta,t-s \rangle \le l_f\|t-s\|_2^2,\]
where $\langle\cdot,\cdot\rangle$ denotes the Euclidean scalar product on $\R^2$.
\item [A3)] The mapping $f$ is linearly bounded in the sense that there exist $\alpha\in L^2(\Omega)$
and $\beta\ge 0$ with
\[\|f(x,s_1,s_2)\|_2 \le \alpha(x) + \beta\|(s_1,s_2)\|_2\quad\text{for a.e.}\ x\in\Omega,\ \forall(s_1,s_2)\in\R^2.\]
\end{itemize}
The weak formulation of (\ref{original:system}) is as follows. A pair $(u_1,u_2)\in H^1_0(\Omega)\times H^1_0(\Omega)$ of functions is called a weak solution of (\ref{original:system}) if 
\begin{equation} \label{variational}
\begin{gathered}
\begin{aligned}
&\exists h_1,h_2\in L^2(\Omega)\ \text{s.t.}\\
&(\nabla u_1,\nabla w)_{L^2} = (h_1,w)_{L^2}\quad\forall w\in H^1_0(\Omega)\\
&(\nabla u_2,\nabla w)_{L^2} = (h_2,w)_{L^2}\quad\forall w\in H^1_0(\Omega)\\
&(h_1(x),h_2(x)) \in f(x,u_1(x),u_2(x))\quad\text{for a.e.}\ x\in\Omega
\end{aligned}
\end{gathered}
\end{equation}
To simplify the notation, we denote $V:=H^1_0(\Omega)\times H^1_0(\Omega)$ and $H:=L^2(\Omega)\times L^2(\Omega)$ 
with $V^*=H^{-1}(\Omega)\times H^{-1}(\Omega)$.
The spaces $V$ and $H$ are equipped with the scalar products
\[(u,v)_V := (u_1,v_1)_{H^1_0} + (u_2,v_2)_{H^1_0},\quad (h,g)_H := (h_1,g_1)_{L^2} + (h_2,g_2)_{L^2}\]
and the corresponding norms.
The duality pairing between $V$ and $V^*$ is given by
\[\langle\phi,u\rangle := \langle\phi_1,u_1\rangle + \langle\phi_2,u_2\rangle.\]
Note that $V\subset H\subset V^*$ is a Gelfand triple.
It can be shown that the set-valued Nemytskii operator given by
\[N_f(u):=\{h\in H: h(x)\in f(x,u(x))\ \text{a.e.}\}\]
is a continuous mapping $N_f:H\rightarrow\mc{CBC}(H)$, which is also $l_f$-ROSL.
If, in addition, the mapping $(s_1,s_2)\mapsto f(x,s_1,s_2)$ is $L_f$-Lipschitz w.r.t.\ the Euclidean norm for all $x\in\Omega$,
then $N_f$ is $L_f$-Lipschitz as well.

We can rewrite \eqref{variational} as an operator inclusion
\begin{equation} \label{operator:inclusion}
0\in (\Delta u_1,\Delta u_2)+N_f(u_1,u_2).
\end{equation}
To comply with the notation in Section \ref{sec:Gelfand}, we denote
$F_V:=(\Delta,\Delta):V\rightarrow V^*$ and $F_H:=N_f:H\rightarrow\mc{CBC}(H)$
with one-sided Lipschitz constants $l_V=-1$ and $l_H=l_f$.
As the embedding $V\subset H$ is compact, the continuity of $N_f$ is, according to Remark \ref{sec:reform-gelf-tripl}(i),
sufficient for the application of Theorem \ref{solvability:hilbert}, provided that $l_f<1/c^2_{VH}$. We stress that, by definition of the embedding constant $c_{VH}$, the quantity $1/c^2_{VH}$ is simply the first Dirichlet eigenvalue of $-\Delta$ on $\Omega$. We have seen in the proof of Theorem \ref{W:theorem} that the norm
\[\|u\|_W^2 = \|u\|_V^2 -l_f\|u\|_H^2\quad\text{for}\ u\in V\]
captures the one-sided properties of the composed mapping $F_V+F_H:V\rightarrow V^*$ in an optimal way
in the sense that it is $l$-ROSL with $l:=-1$ w.r.t.\ $\|\cdot\|_W$.
We find that
\begin{align*}
&e_{W^*}(F_H(u),F_H(\tilde u)) 
\le \tfrac{1}{\sqrt{\tfrac{1}{c_{VH}^2}-l_H}}e_H(N_f(u),N_f(\tilde u))\\
&\le \tfrac{L_f}{\sqrt{\tfrac{1}{c_{VH}^2}-l_H}}\|u-\tilde u\|_H 
\le \tfrac{L_f}{\tfrac{1}{c_{VH}^2}-l_H}\|u-\tilde u\|_W
= \tfrac{c_{VH}^2L_f}{1-c_{VH}^2l_H}\|u-\tilde u\|_W
\end{align*}
for all $u,\tilde u\in V$.
In order to check the assumptions for the iterative algorithm \eqref{algorithm:formulation}, we distinguish two cases.\\
{\bf Case 1:}  $l_f\le 0$. In this case 
$F_V:(V,\|\cdot\|_W)\rightarrow(V^*,\|\cdot\|_{W^*})$ is $1$-Lipschitz, and thus $F_V+F_H:(V,\|\cdot\|_W)\rightarrow(V^*,\|\cdot\|_{W^*})$ is $L$-Lipschitz 
with $L:=1+\tfrac{c_{VH}^2L_f}{1-c_{VH}^2l_f}$.
Thus we can compute a solution of system \eqref{original:system} or, equivalently, operator inclusion
\eqref{operator:inclusion} by applying the iterative algorithm \eqref{algorithm:formulation},
provided that $L< -2l=2$, or, equivalently, 
\[L_f<\tfrac{1}{c_{VH}^2}-l_f.\] 
{\bf Case 2:}  $l_f \in [0,\frac{1}{c_{VH}^2})$. In this case it follows from the estimates in Remark~\ref{sec:reform-gelf-tripl}(ii) that $F_V:(V,\|\cdot\|_W)\rightarrow(V^*,\|\cdot\|_{W^*})$ is $\frac{1}{1-c_{VH}^2 l_f}$-Lipschitz, and thus $F_V+F_H:(V,\|\cdot\|_W)\rightarrow(V^*,\|\cdot\|_{W^*})$ is $\tfrac{1+ c_{VH}^2L_f}{1-c_{VH}^2l_f}$-Lipschitz.
As a consequence, the iterative algorithm \eqref{algorithm:formulation} applies in this case if
\[
L_f< \tfrac{1}{c_{VH}^2}-2 l_f,
\] 
which a posteriori requires $l_f < \frac{1}{3c_{VH}^2}$ since $l_f \le L_f$.

\subsection{Computational considerations}

We are looking for a solution $u\in V$ to the operator inclusion \eqref{operator:inclusion}.
Given any initial value $u_0\in V$, the numerical routine proposed in \eqref{algorithm:formulation} 
consecutively constructs a sequence $(u_n)_n\subset V$ of approximate solutions that converge to
a solution of \eqref{operator:inclusion}. 
In the present context, the iteration reads
\begin{align}
u_{n+1} &= u_n-\tfrac{1}{2l}J_W\Proj_{W^*}(0,F_V(u_n)+F_H(u_n)) \nonumber\\
&= u_n-\tfrac{1}{2l}\Proj_W(0,J_WF_V(u_n)+J_WN_f(u_n)),\label{eq:comp-cons-iteration}
\end{align}
because $J_W:(V^*,\|\cdot\|_{W^*})\rightarrow(V,\|\cdot\|_W)$ is an isometrical isomorphism.
From a computational perspective, it may be advantageous to recast the optimization problem
\begin{align*}
\|J_WF_V(u_n)+J_Wh\|_W = \min!\quad\text{subject to}\quad h\in N_f(u_n)
\end{align*}
with pointwise inequality constraints into an unconstrained dual problem.

\begin{lemma} \label{unique:minimizer}
Let $X$ be a Hilbert space with inner product $(\cdot,\cdot)_X$ and norm $\|\cdot\|_X$,
and let $A\in\mc{CBC}(X)$. 
Then the optimization problems
\begin{equation} \label{op}
\tfrac12\|x\|_X^2 = \min!\quad\text{subject to}\quad x\in A
\end{equation}
and 
\begin{equation}\label{dp}
\tfrac12\|x\|_X^2 + \sigma_X(-x,A) = \min!
\end{equation}
on the entire space $X$ possess the same unique solution.
\end{lemma}
\begin{proof}
It is well-known (see, e.g., \cite[Proposition 7.4]{Clarke:13}) that \eqref{op} possesses a unique
solution $x^*\in A$, which is also the unique solution of the variational inequality 
\begin{equation}\label{vip}
(x,a-x)_X \ge 0\quad\forall a\in A
\end{equation}
in $A$. 
As the function $x \mapsto \tfrac12\|x\|_X^2 + \sigma_X(-x,A)$ is convex, lower semicontinuous and coercive,
problem \eqref{dp} possesses at least one solution.
Any solution $x_*\in X$ of \eqref{dp} satisfies the necessary optimality condition
\[0\in x_*-\partial\sigma_X(-x_*,A),\]
where $\partial\sigma_X(\cdot,\cdot)$ denotes the subdifferential w.r.t.\ the first variable.
This implies
\begin{equation}\label{local:4}
x_*\in\partial\sigma_X(-x_*,A)=\argmax_{a\in A}(-x_*,a)_X\subset A,
\end{equation}
so that $x_*\in A$.
Moreover, \eqref{local:4} implies 
\[(-x_*,x_*)_X \ge (-x_*,a)_X\quad\text{for all}\ a\in A,\]
so that $x_*$ solves \eqref{vip} and hence $x_*=x^*$.
\end{proof}

In the situation of our example, the support function in the dual problem \eqref{dp}
can be computed explicitly.
\begin{lemma} \label{Nemytskii:representation}
For all $u,v\in V$ we have
\[\sigma_W(v,J_WN_f(u)) = \int_\Omega\sigma(v(x),f(x,u(x))dx,\]
where $\sigma(\cdot,\cdot)$ denotes the support function on $\R^2$.
\end{lemma}
\begin{proof}
Given $u,v\in V$, we construct some $h_v\in N_f(u)$ such that
\begin{equation} \label{maximizer}
(v,h_v)_H=\max_{h\in N_f(u)}(v,h)_H.
\end{equation}
Theorem 8.2.11 in \cite{Aubin:Frankowska:90} on marginal maps ensures that the multivalued mapping
$G:\Omega\rightarrow\mc{CBC}(\R^2)$ given by
\[G(x):=\{t\in f(x,u(x)):\langle v(x),t\rangle = \max_{s\in f(x,u(x))}\langle v(x),s\rangle\}\neq\emptyset\]
is measurable, and \cite[Theorem 8.1.3]{Aubin:Frankowska:90} ensures that $G$ possesses a measurable selection 
$h_v:\Omega\rightarrow\R^2$, i.e.
\[h_v(x)\in G(x)\subset f(x,u(x))\quad\text{for a.e.}\ x\in\Omega.\]
The linear growth bound A3) ensures that $h_v\in H$, and therefore $h_v\in N_f(u_n)$.
By monotonicity of the integral and the construction of $h_v$, we have
\[(v,h_v)_H - (v,h)_H = \int_\Omega\langle v(x),h_v(x)-h(x)\rangle dx \ge 0\]
for all $h\in N_f(u)$, so that $h_v$ satisfies condition \eqref{maximizer}. By construction of $h_v$, we find
\begin{align*}
\sigma_W(v,J_WN_f(u)) &= \sup_{h\in N_f(u)}(v,J_Wh)_W = \sup_{h\in N_f(u)}(v,h)_H\\
&= (v,h_v)_H = \int_\Omega\sigma(v(x),f(x,u(x))dx.
\end{align*}
\end{proof}


The following statement is a consequence of Lemmas \ref{unique:minimizer} and
\ref{Nemytskii:representation}.

\begin{corollary}
\label{comp-cons-corollary}  
In the iteration defined by (\ref{eq:comp-cons-iteration}), the function $-2l(u_{n+1}-u_n) \in V$ is the unique minimizer 
of the functional $I:V \to \R$ given by 
\begin{align*}
I(h) = &\frac{1}{2}\|h\|_W^2 - \langle F_V(u_n),h\rangle + \int_\Omega\sigma(-h(x),f(x,u_n(x))dx\\
=&  \int_{\Omega} \Bigl[\frac{1}{2} \sum_{i=1}^2 \Bigl(|\nabla h_i|^2 -l_f |h_i|^2
+2  \nabla h_i \cdot \nabla u_{n,i}\Bigr)+\sigma(-h(x),f(x,u_n(x))\Bigr]dx.
\end{align*}
\end{corollary}

\subsection{Some numerical results}

We first consider the problem
\begin{equation}\label{bsp:1:incl}
(-\Delta u_1,-\Delta u_2) \in -\frac49\cdot\frac{|u|^2}{1+|u|^2}u+l_fu + B_R(0) 
\end{equation}
on $\Omega=(0,1)$ with $u\equiv 0$ on $\partial\Omega$ and $l_f\le 0$.
Elementary computations show that the right-hand side is $l_f$-ROSL and $L_f$-Lipschitz with 
a constant $L_f=\tfrac12-l_f$.
According to Section \ref{example:theory} the composed mapping 
\[F_V+F_H:(V,\|\cdot\|_W)\rightarrow(V^*,\|\cdot\|_{W^*})\] 
is $l$-ROSL and $L$-Lipschitz with constants $l=-1$ and 
\[L=1+\tfrac{L_f}{1/c_{VH}^2-l_f}=1+\tfrac{1/2-l_f}{\pi^2-l_f}<2,\]
so that algorithm \eqref{algorithm:formulation} is applicable with a theoretical speed of convergence
\[\kappa := -L/2l = \tfrac12(1+\tfrac{1/2-l_f}{\pi^2-l_f})<1.\]
The results for parameters $l_f=-1$ and $R=10$ and initial values 
\[u_{01}(x)=\tfrac12\sin(2\pi x),\quad u_{02}=\tfrac12\sin(16\pi x)\] 
displayed in Figure \ref{bsp:1} show that the bound \eqref{dist:set}
is realistic in this case. 
The residual 
\[r_n=\dist_{W^*}(0,\Delta u_n+N_F(u_n))\] 
is approximately halved in every 
iteration, while the theoretical bound guarantees a reduction by a factor 
\[\kappa = \tfrac12 + \tfrac{3}{4(\pi^2+1)} \approx 0.569.\]

\begin{figure}
\begin{tabular}{cc}
\begin{minipage}{2.3cm}
\begin{footnotesize}
\begin{tabular}[b]{cc} 
    \toprule
    steps & residual \\ 
    \midrule
    0 & 17.506 \\ 1 & 8.8020\\ 2 & 4.4425\\ 3 & 2.2531\\ 4 & 1.1496\\ 5 & 0.5899\\ 6 & 0.3040\\ 7 & 0.1571\\ 8 & 0.0815 \\
    \bottomrule
\end{tabular}
\end{footnotesize}
\end{minipage}
&
\begin{minipage}{8cm}
\begin{scriptsize}
\psfrag{0}{0}
\psfrag{0.5}{0.5}
\psfrag{-0.5}{-0.5}
\psfrag{1}{1}
\psfrag{u0}{$u_0$}
\psfrag{u1}{$u_1$}
\psfrag{u2}{$u_2$}
\psfrag{u3}{$u_3$}
\psfrag{u4}{$u_4$}
\psfrag{u5}{$u_5$}
\psfrag{u6}{$u_6$}
\psfrag{u7}{$u_7$}
\psfrag{u8}{$u_8$}
\psfrag{O}{}
\includegraphics[scale=0.8]{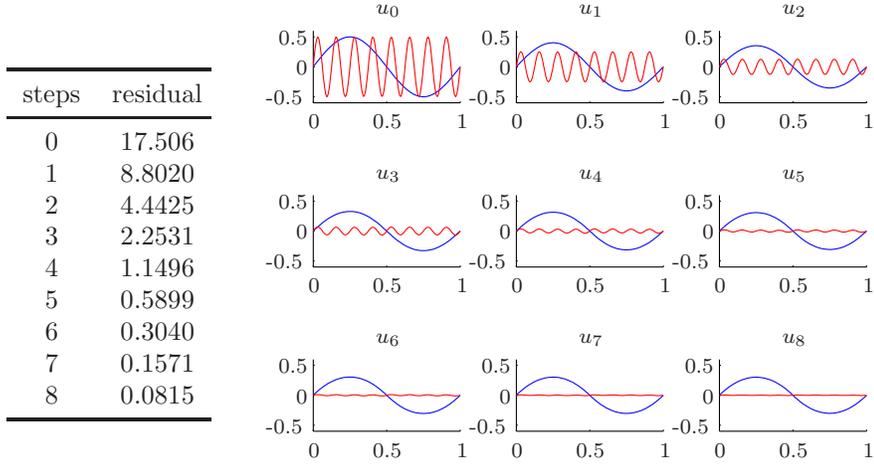}
\end{scriptsize}
\end{minipage}
\end{tabular}
\caption{Iterates of the numerical algorithm \eqref{algorithm:formulation} applied to inclusion \eqref{bsp:1:incl}.
The residual is measured in the norm $\|\cdot\|_{W^*}$. 
\label{bsp:1}}
\end{figure}

\medskip

We now consider the problem
\begin{equation}\label{bsp:2:incl}
(-\Delta u_1,-\Delta u_2) \in (-u_1u_2+1-u_1+x,-u_1u_2+1-u_2+x) + B_R(0)
\end{equation}
on $\Omega=(0,1)$ with $u\equiv 0$ on $\partial\Omega$. 
The application of algorithm \eqref{algorithm:formulation} to this problem  is a-priori not theoretically justified.
Examples not included here show that the iteration can indeed diverge.
For most nonnegative initial values of moderate magnitude, however, the iteration converges as depicted
in Figure \ref{bsp:2}, where $R=5$ and 
\[u_{01}(x)=x(1-x)e^{-\frac{(x-0.1)^2}{0.1}},\quad u_{02}(x)=x(1-x)e^{-\frac{(x-0.8)^2}{0.01}},\quad x\in[0,1].\]
In this case, the decay of the residuals justifies a-posteriori the use of the
method and the validity of the result.

\begin{figure}
\begin{tabular}{cc}
\begin{minipage}{2.3cm}
\begin{footnotesize}
\begin{tabular}[b]{cc} 
    \toprule
    steps & residual\\ 
    \midrule
    0 & 0.6516\\ 1 & 0.3375\\ 2 & 0.1802\\ 3 & 0.1000\\ 4 & 0.0584\\ 5 & 0.0363\\ 6 & 0.0242\\ 7 & 0.0171\\ 8 & 0.0127\\
    \bottomrule
\end{tabular}
\end{footnotesize}
\end{minipage}
&
\begin{minipage}{8cm}
\begin{scriptsize}
\psfrag{0}{0}
\psfrag{0.05}{}
\psfrag{0.1}{0.1}
\psfrag{0.15}{}
\psfrag{0.5}{0.5}
\psfrag{1}{1}
\psfrag{u0}{$u_0$}
\psfrag{u1}{$u_1$}
\psfrag{u2}{$u_2$}
\psfrag{u3}{$u_3$}
\psfrag{u4}{$u_4$}
\psfrag{u5}{$u_5$}
\psfrag{u6}{$u_6$}
\psfrag{u7}{$u_7$}
\psfrag{u8}{$u_8$}
\psfrag{O}{}
\includegraphics[scale=0.8]{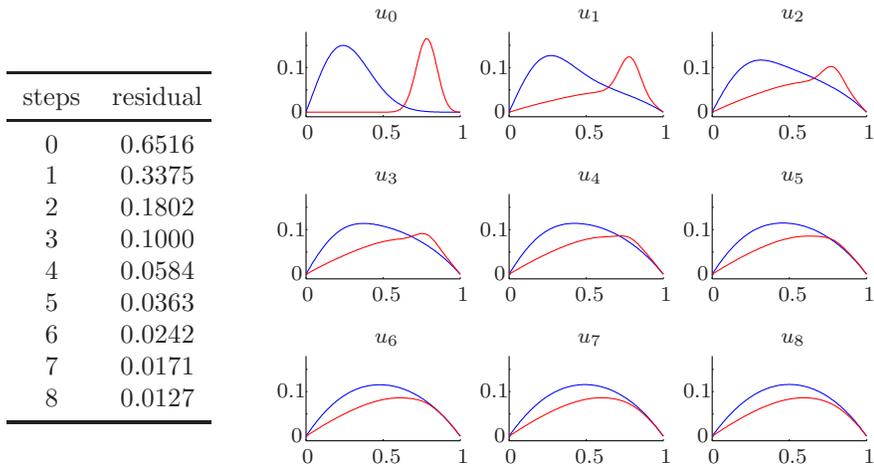}
\end{scriptsize}
\end{minipage}
\end{tabular}
\caption{Iterates of the numerical algorithm \eqref{algorithm:formulation} applied to inclusion \eqref{bsp:2:incl}.
The residual is measured in the norm $\|\cdot\|_{W^*}$. 
\label{bsp:2}}
\end{figure}

\bibliographystyle{plain}
\bibliography{standard}

\begin{thebibliography}{10}

\bibitem{Aubin:Frankowska:90}
J.-P. Aubin and H.~Frankowska.
\newblock {\em Set-Valued Analysis}.
\newblock Birk\-h\"au\-ser, Boston, 1990.

\bibitem{Baier:Farkhi:07}
R.~Baier and E.~Farkhi.
\newblock Regularity and integration of set-valued maps represented by
  generalized {S}teiner points.
\newblock {\em Set-Valued Anal.}, 15(2):185--207, 2007.

\bibitem{Beyn:Rieger:10}
W.-J. Beyn and J.~Rieger.
\newblock The implicit {E}uler scheme for one-sided {L}ipschitz differential
  inclusions.
\newblock {\em Discrete Contin.~Dyn.~Syst.~Ser.~B}, 14:409--428, 2010.

\bibitem{Beyn:Rieger:12}
W.-J. Beyn and J.~Rieger.
\newblock Galerkin finite element methods for semilinear elliptic differential
  inclusions.
\newblock {\em Discrete Contin.~Dyn.~Syst.~Ser.~B}, 18(2):295--312, 2013.

\bibitem{Beyn:Rieger:14}
W.-J. Beyn and J.~Rieger.
\newblock An iterative method for solving relaxed one-sided {L}ipschitz
  algebraic inclusions.
\newblock {\em J.~Optim.~Theory Appl.}, to appear.

\bibitem{Carl:Le:Motreanu:07}
S.~Carl, V.K. Le, and D.~Motreanu.
\newblock {\em Nonsmooth variational problems and their inequalities}.
\newblock Springer Monographs in Mathematics. Springer, New York, 2007.

\bibitem{Cheng:Cong:Xue:11}
Y.~Cheng, F.~Cong, and X.~Xue.
\newblock Boundary value problems of a class of nonlinear partial differential
  inclusions.
\newblock {\em Nonlinear Anal. Real World Appl.}, 12(6):3095--3102, 2011.

\bibitem{Clarke:13}
F.~Clarke.
\newblock {\em Functional Analysis, Calculus of Variations and Optimal
  Control}.
\newblock Graduate Texts in Mathematics 264, Springer, London, 2013.

\bibitem{Donchev:96}
T.~Donchev.
\newblock Qualitative properties of a class of differential inclusions.
\newblock {\em Glas.~Mat.~Ser.~III}, 31(51)(2):269--276, 1996.

\bibitem{Donchev:02}
T.~Donchev.
\newblock Properties of one sided {L}ipschitz multivalued maps.
\newblock {\em Nonlinear Anal.}, 49:13--20, 2002.

\bibitem{Donchev:04}
T.~Donchev.
\newblock Surjectivity and fixed points of relaxed dissipative multifunctions.
  differential inclusions approach.
\newblock {\em J.~Math.~Anal.~Appl.}, 299:525--533, 2004.

\bibitem{Kenderov:76}
P.~Kenderov.
\newblock Multivalued monotone mappings are almost everywhere single-valued.
\newblock {\em Studia Math.}, 56(3):199--203, 1976.

\bibitem{Mordukhovich:Tian:14}
B.S. Mordukhovich and Y.~Tian.
\newblock Implicit {E}uler approximation and optimization of one-sided
  {L}ipschitzian differential inclusions.
\newblock {\em CoRR}, http://arxiv.org/abs/1410.2207, 2014.

\bibitem{Rockafellar:70}
R.T. Rockafellar.
\newblock {\em Convex Analysis}, volume~28 of {\em Princeton mathematical
  series}.
\newblock Princeton University Press, 1970.

\bibitem{Rockafellar:76}
R.T. Rockafellar.
\newblock Monotone operators and the proximal point algorithm.
\newblock {\em SIAM J.~Control Optimization}, 14(5):877--898, 1976.

\end{thebibliography}

\end{document}